\documentclass[11pt,reqno]{amsart}
\usepackage{latexsym}
\usepackage{graphicx}
\usepackage{fancyhdr,amssymb}
\usepackage{amsmath,mathpazo}
\usepackage{amssymb}
\usepackage{relsize}
\usepackage[normalem]{ulem}

\theoremstyle{plain}
\numberwithin{equation}{section}
\newtheorem{thm}{Theorem}[section]
\newtheorem{theorem}[thm]{Theorem}

\theoremstyle{definition}

\begin{document}
\fancyhead{}
\renewcommand{\headrulewidth}{0pt}
\fancyfoot{}
\fancyfoot[LE,RO]{\medskip \thepage}

\setcounter{page}{1}

\title[Relationship Between the Prime-Counting Function and a Unique Prime Number Sequence]{Relationship Between the Prime-Counting Function and a Unique Prime Number Sequence}
\author{Michael P. May}
\address{443 W 2825 S\\
         Perry, Utah\\
         84302}
\email{mikemay@mst.edu \textit{or} mikemaybbi@gmail.com}

\begin{abstract}

In mathematics, the prime counting function $\pi(x)$ is defined as the function yielding the number of primes less than or equal to a given number $x$. In this paper, we prove that the asymptotic limit of a summation operation performed on a unique subsequence of the prime numbers yields the prime number counting function $\pi(x)$ as $x$ approaches $\infty$. We also show that the prime number count $\pi(n)$ can be estimated with a notable degree of accuracy by performing the summation operation on the subsquence up to a limit $n$.
 
\end{abstract}

\maketitle
\section{GENERATING $\mathnormal{\mathbb{P{'}}}$ AND $\mathnormal{\mathbb{P{''}}}$}

\noindent Consider the prime number subsequence of higher order \cite{OEIS_A333242}

\begin{equation*}
\mathbb{P^{'}} = \left\lbrace {p{'}} \right\rbrace = \left\lbrace 2,5,7,13,19,23,29,31,37,43,47,53,59,61,71,... \right\rbrace.
\end{equation*}

\ 

\noindent It was discovered \cite{HigherOrderPrimeNumberSequences} that $\mathbb{P{'}}$ can be generated via an alternating sum of the prime number subsequences of increasing order, i.e., 

\

\begin{equation}
\mathnormal{\mathbb{P^{'}}={\left\lbrace{(-1)^{n-1}}\left\lbrace{p^{(n)}}\right\rbrace\right\rbrace}_{n=1}^\infty} \label{eq:1}
\end{equation}

\

\noindent where the right-hand side of Eq. \ref{eq:1} is an expression of the alternating sum 

\ \\
\noindent\rule{0.84in}{0.4pt} \par
\medskip
\indent\indent {\fontsize{8pt}{9pt} \selectfont MSC2020: 11A41, 11B05, 11K31 \par}
\indent\indent {\fontsize{8pt}{9pt} \selectfont Key words and phrases: Prime-counting function, Prime numbers, \par}
\indent\indent {\fontsize{8pt}{9pt} \selectfont Higher-order prime number sequences, Prime gaps \par}

\thispagestyle{fancy}

\vfil\eject
\fancyhead{}
\fancyhead[CO]{\hfill RELATIONSHIP OF $\mathnormal{\pi{(x)}}$ TO A PRIME SUBSEQUENCE}
\fancyhead[CE]{M. P. MAY  \hfill}
\renewcommand{\headrulewidth}{0pt}   

\begin{equation}
\left\lbrace{p^{(1)}}\right\rbrace - \left\lbrace{p^{(2)}}\right\rbrace + \left\lbrace{p^{(3)}}\right\rbrace - \left\lbrace{p^{(4)}}\right\rbrace + \left\lbrace{p^{(5)}}\right\rbrace -... \;. \label{eq:2}
\end{equation}

\

\noindent The prime number subsequences of increasing order \cite{PrimesHavingPrimeSubscripts} in Expression \ref{eq:2} are defined as

\begin{equation*}
\left\lbrace{p^{(1)}}\right\rbrace = {\left\lbrace{p_n}\right\rbrace}_{n=1}^\infty = \left\lbrace2,3,5,7,11,13,17,19,23,29,31,37,41,43,...\right\rbrace 
\end{equation*}

\

\begin{equation*}
\left\lbrace{p^{(2)}}\right\rbrace = {\left\lbrace{p_{p_n}}\right\rbrace}_{n=1}^\infty = \left\lbrace3,5,11,17,31,41,59,67,83,109,127,...\right\rbrace 
\end{equation*}

\

\begin{equation*}
\left\lbrace{p^{(3)}}\right\rbrace = {\left\lbrace{p_{p_{p_n}}}\right\rbrace}_{n=1}^\infty = \left\lbrace5,11,31,59,127,179,277,331,...\right\rbrace 
\end{equation*}

\

\begin{equation*}
\left\lbrace{p^{(4)}}\right\rbrace = {\left\lbrace{p_{p_{p_{p_n}}}}\right\rbrace}_{n=1}^\infty = \left\lbrace11,31,127,277,709,...\right\rbrace 
\end{equation*}

\

\begin{equation*}
\left\lbrace{p^{(5)}}\right\rbrace = {\left\lbrace{p_{p_{p_{p_{p_n}}}}}\right\rbrace}_{n=1}^\infty = \left\lbrace31,127,709,...\right\rbrace 
\end{equation*}

\ \\

\noindent and so on and so forth. It is noted for clarification that the operation performed on the right-hand side of Eq. \ref{eq:1} denotes an alternating sum of the \underline{entire sets} of prime number subsequences of increasing order.

\

\noindent The prime number subsequence of higher order $\mathbb{P^{'}}$ can also be generated by performing an alternating sum of the individual elements across the sets. To illustrate this, we arrange the subsequences in Expression \ref{eq:2} side-by-side and sum elements laterally across the rows to create the new $\mathbb{P^{'}}$ subsequence term-by-term as follows:

\newpage

\begin{table}[ht]
\centering
\begin{tabular}{cccccccccc}

(row) & $+p^{(1)}$ & $-p^{(2)}$ & $+p^{(3)}$ & $-p^{(4)}$ & $+p^{(5)}$ & $-p^{(6)}$ & $...$ & $p{'}$ \\ 
&  &  &  &  &  &  &  &  &  \\ 
(1) & 2 & $\longrightarrow$ & $\longrightarrow$ & $\longrightarrow$ & $\longrightarrow$ & $\longrightarrow$ & $\longrightarrow$ &  2 \\ 
(2) & 3 & 3 & $\longrightarrow$ & $\longrightarrow$ & $\longrightarrow$ & $\longrightarrow$ & $\longrightarrow$ &  0 \\ 
(3) & 5 & 5 & 5 & $\longrightarrow$ & $\longrightarrow$ & $\longrightarrow$ & $\longrightarrow$ &  5 \\ 
(4) & 7 & $\longrightarrow$ & $\longrightarrow$ & $\longrightarrow$ & $\longrightarrow$ & $\longrightarrow$ & $\longrightarrow$ &  7 \\ 
(5) & 11 & 11 & 11 & 11 & $\longrightarrow$ & $\longrightarrow$ & $\longrightarrow$ &  0 \\ 
(6) & 13 & $\longrightarrow$ & $\longrightarrow$ & $\longrightarrow$ & $\longrightarrow$ & $\longrightarrow$ & $\longrightarrow$ &  13 \\ 
(7) & 17 & 17 & $\longrightarrow$ & $\longrightarrow$ & $\longrightarrow$ & $\longrightarrow$ & $\longrightarrow$ &  0 \\ 
(8) & 19 & $\longrightarrow$ & $\longrightarrow$ & $\longrightarrow$ & $\longrightarrow$ & $\longrightarrow$ & $\longrightarrow$ &  19 \\ 
(9) & 23 & $\longrightarrow$ & $\longrightarrow$ & $\longrightarrow$ & $\longrightarrow$ & $\longrightarrow$ & $\longrightarrow$ &  23 \\ 
(10) & 29 & $\longrightarrow$ & $\longrightarrow$ & $\longrightarrow$ & $\longrightarrow$ & $\longrightarrow$ & $\longrightarrow$ &  29 \\ 
(11) & 31 & 31 & 31 & 31 & 31 & $\longrightarrow$ & $\longrightarrow$ &  31 \\ 
$\vdots$ & $\vdots$ & $\vdots$ & $\vdots$ & $\vdots$ & $\vdots$ & $\vdots$ & $\vdots$ & $\vdots$  
\end{tabular}
\caption{Alternating Sum of $p^{(n)}$}
\label{table:sum}
\end{table}

\noindent Thus, the infinite prime number subsequence $\mathbb{P^{'}}$ of higher order \cite{OEIS_A333242} emerges in the rightmost column of Table \ref{table:sum}:

\begin{equation*}
\mathnormal{\mathbb{P^{'}} = \left\lbrace {p{'}} \right\rbrace} = \left\lbrace 2,5,7,13,19,23,29,31,37,43,47,53,59,61,71,... \right\rbrace.
\end{equation*}

\

\noindent The prime number subsequence of higher order $\mathbb{P^{'}}$ can also be generated by the N-sieve \cite{HigherOrderPrimeNumberSequences}. We now demonstrate how that is accomplished. Starting with $n=1$, choose the prime number with subscript 1 (i.e., $p_1=2$) as the first term of the subsequence and eliminate that prime number from the natural number line.  Then, proceed forward on $\mathbb{N}$ from $1$ to the next available natural number.  Since $2$ was eliminated from the natural number line in the previous step, one moves forward to the next available natural number that has not been eliminated, which is $3$.  $3$ then becomes the subscript for the next $\mathbb{P^{'}}$ term which is $p_3=5$, and $5$ is then eliminated from the natural number line, and so on and so forth.  Such a sieving operation has been carried out in the chart below for the natural numbers $1$ to $100$. \\

\

\begin{center}
\noindent 1 \raisebox{-2.5pt}{\Huge\textcircled{\raisebox{2pt}{\normalsize 2}}} 3 4 \raisebox{-2.5pt}{\Huge\textcircled{\raisebox{2pt}{\normalsize 5}}} 6 \raisebox{-2.5pt}{\Huge\textcircled{\raisebox{2pt}{\normalsize 7}}} 8 9 10 11 12 \raisebox{-2.5pt}{\Huge\textcircled{\raisebox{2pt}{\normalsize 13}}} 14 15 16 17 18 \raisebox{-2.5pt}{\Huge\textcircled{\raisebox{2pt}{\normalsize 19}}} 20  
\end{center}

\noindent

\begin{center}
\noindent 21 22 \raisebox{-2.5pt}{\Huge\textcircled{\raisebox{2pt}{\normalsize 23}}} 24 25 26 27 28 \raisebox{-2.5pt}{\Huge\textcircled{\raisebox{2pt}{\normalsize 29}}} 30 \raisebox{-2.5pt}{\Huge\textcircled{\raisebox{2pt}{\normalsize 31}}} 32 33 34 35 36 \raisebox{-2.5pt}{\Huge\textcircled{\raisebox{2pt}{\normalsize 37}}} 38 39 40  
\end{center}

\noindent

\begin{center}
\noindent 41 42 \raisebox{-2.5pt}{\Huge\textcircled{\raisebox{2pt}{\normalsize 43}}} 44 45 46 \raisebox{-2.5pt}{\Huge\textcircled{\raisebox{2pt}{\normalsize 47}}} 48 49 50 51 52 \raisebox{-2.5pt}{\Huge\textcircled{\raisebox{2pt}{\normalsize 53}}} 54 55 56 57 58 \raisebox{-2.5pt}{\Huge\textcircled{\raisebox{2pt}{\normalsize 59}}} 60 
\end{center}

\noindent

\begin{center}
\noindent \raisebox{-2.5pt}{\Huge\textcircled{\raisebox{2pt}{\normalsize 61}}} 62 63 64 65 66 67 68 69 70 \raisebox{-2.5pt}{\Huge\textcircled{\raisebox{2pt}{\normalsize 71}}} 72 \raisebox{-2.5pt}{\Huge\textcircled{\raisebox{2pt}{\normalsize 73}}} 74 75 76 77 78 \raisebox{-2.5pt}{\Huge\textcircled{\raisebox{2pt}{\normalsize 79}}} 80 
\end{center}

\noindent

\begin{center}
\noindent  81 82 83 84 85 86 87 88 \raisebox{-2.5pt}{\Huge\textcircled{\raisebox{2pt}{\normalsize 89}}} 90 91 92 93 94 95 96 \raisebox{-2.5pt}{\Huge\textcircled{\raisebox{2pt}{\normalsize 97}}} 98 99 100 
\end{center}

\

\noindent Thus, we may optionally designate $\mathbb{P{'}}$, which has been created via the N-sieve operation above, by the following notation \cite{HigherOrderPrimeNumberSequences} to indicate that the natural numbers $\mathbb{N}$ have been sieved to produce this prime number subsequence:

\begin{equation*}
\mathnormal{\left\lfloor{\raisebox{-.3pt}{\!{\dashuline{\begin{math}\,\mathbb{N}\end{math}}}}}\right\rfloor} = \mathnormal{\mathbb{P{'}}} = \left\lbrace{2,5,7,13,19,23,29,31,37,43,47,53,59,61,71,...}\right\rbrace.
\end{equation*}

\noindent 

\noindent Regardless of the method used to generate $\mathbb{P{'}}$, when the prime numbers in this unique subsequence are applied as indexes to the set of all prime numbers $\mathbb{P^{}}$, one obtains the next higher-order prime number subsequence 

\begin{equation*}
\mathnormal{\mathbb{P^{''}} = \left\lbrace {p{''}} \right\rbrace} = \left\lbrace3,11,17,41,67,83,109,127,157,191,211,241,...\right\rbrace. 
\end{equation*}

\noindent

\noindent By definition \cite{HigherOrderPrimeNumberSequences}, the sequence $\mathbb{P{''}}$ can be generated via the expression

\noindent

\begin{equation}
\mathnormal{\mathbb{P^{''}}={\left\lbrace{(-1)^{n}}\left\lbrace{p^{(n)}}\right\rbrace\right\rbrace}_{n=2}^\infty} \label{eq:3}
\end{equation}

\noindent

\noindent where an expansion of the right-hand side of Eq. \ref{eq:3} is the alternating sum 

\begin{equation}
\left\lbrace{p^{(2)}}\right\rbrace - \left\lbrace{p^{(3)}}\right\rbrace + \left\lbrace{p^{(4)}}\right\rbrace - \left\lbrace{p^{(5)}}\right\rbrace + \left\lbrace{p^{(6)}}\right\rbrace -... \;. \label{eq:4}
\end{equation}

\noindent Further, it has been shown \cite{HigherOrderPrimeNumberSequences} that the subsequences $\mathbb{P^{'}}$ and $\mathbb{P^{''}}$ when added together form the entire set of prime numbers $\mathbb{P^{}}$:

\begin{equation}
\mathnormal{\mathbb{P{}}}=\mathnormal{\mathbb{P{'}}+\mathbb{P{''}}}. \label{eq:5}
\end{equation}

\noindent

\noindent We sketch a short proof of Eq. \ref{eq:5} here:

\noindent

\begin{proof}

\noindent 

\

\noindent It has been established \cite{HigherOrderPrimeNumberSequences} that

\begin{equation*}
\mathnormal{\mathbb{P{'}}} = {\left\lbrace{(-1)^{n-1}}\left\lbrace{p^{(n)}}\right\rbrace\right\rbrace}_{n=1}^\infty = \left\lbrace{p^{(1)}}\right\rbrace - \left\lbrace{p^{(2)}}\right\rbrace + \left\lbrace{p^{(3)}}\right\rbrace - ... 
\end{equation*}

\noindent 

\noindent and

\begin{equation*}
\mathnormal{\mathbb{P{''}}} = {\left\lbrace{(-1)^{n}}\left\lbrace{p^{(n)}}\right\rbrace\right\rbrace}_{n=2}^\infty = \left\lbrace{p^{(2)}}\right\rbrace - \left\lbrace{p^{(3)}}\right\rbrace + \left\lbrace{p^{(4)}}\right\rbrace - ... \; .
\end{equation*}

\ \ \

\noindent Then,

\begin{align*}
\mathnormal{\mathbb{P{'}}+\mathbb{P{''}}} = & \left\lbrace{p^{(1)}}\right\rbrace - \left\lbrace{p^{(2)}}\right\rbrace + \left\lbrace{p^{(3)}}\right\rbrace - ... \\
& \;\;\;\;\;\;\;\;\;\;\;\;\;\;\;\;\;\;\;\; + \\
& \left\lbrace{p^{(2)}}\right\rbrace - \left\lbrace{p^{(3)}}\right\rbrace + \left\lbrace{p^{(4)}} \right\rbrace - ... = \left\lbrace{p^{(1)}}\right\rbrace = \mathnormal{\mathbb{P{}}}.
\end{align*}

\end{proof}

\noindent An interesting property was observed in the relationship between the set of all prime numbers $\mathbb{P{}}$ and the complement prime number sets $\mathbb{P{'}}$ and $\mathbb{P{''}}$. Since $\mathbb{P{''}}=\mathbb{P}_{\mathbb{P{'}}}$, Eq. \ref{eq:5} can be restated as

\begin{align*}
\mathnormal{\mathbb{P{''}}} & \mathnormal{=\mathbb{P{}} - \left\lbrace2,5,7,13,19,23,29,...\right\rbrace} \\
& \mathnormal{=\left\lbrace {p_{{}_{{}_{\mathnormal{{2}}}}},p_{{}_{{}_{{\mathnormal{5}}}}},p_{{}_{{}_{{\mathnormal{7}}}}},p_{{}_{{}_{{\mathnormal{13}}}}},p_{{}_{{}_{{\mathnormal{19}}}}},p_{{}_{{}_{{\mathnormal{23}}}}}, p_{{}_{{}_{{\mathnormal{29}}}}},...}\right\rbrace = \mathbb{P}_{\mathbb{P{'}}}}
\end{align*}

\noindent 

\noindent where the prime numbers of the subsequence $\mathbb{P{'}}$ form the indexes for the complement set of primes $\mathbb{P{''}}$ such that 

\begin{equation*}
\mathnormal{\mathnormal{\mathbb{P{''}}} = \mathnormal{\mathbb{P}_{{\mathbb{P{'}}}} = \lbrace p_k \mid k \in \mathbb{P{'}} \rbrace}}.
\end{equation*}

\noindent{}

\noindent We now calculate the average gap size for $\mathbb{P{'}}$ at $\infty$.

\noindent 

\section{AVERAGE GAP OF $\mathnormal{\mathbb{P{'}}}$}

\noindent We will now derive the asymptotic density for the prime number subsequence $\mathbb{P{'}}$ assuming that $1/\ln{n}$ is the asymptotic density of the set of all prime numbers $\mathbb{P{}}$ at $\infty$. We approach this task via alternately adding and subtracting the prime number densities (or ``probabilities" as they are also called) of the prime number subsequences of increasing order to arrive at a value for the density of $\mathbb{P{'}}$.  We begin by recalling \cite{HigherOrderPrimeNumberSequences} that the prime number subsequence $\mathbb{P{'}}$ is formed by the alternating series

\begin{equation*}
\mathnormal{\mathbb{P{'}}} = {\left\lbrace{(-1)^{n-1}}\left\lbrace{p^{(n)}}\right\rbrace\right\rbrace}_{n=1}^\infty = \left\lbrace{p^{(1)}}\right\rbrace - \left\lbrace{p^{(2)}}\right\rbrace + \left\lbrace{p^{(3)}}\right\rbrace -  ... 
\end{equation*}

\noindent 

\noindent where

\begin{equation*}
\left\lbrace{p^{(k)}}\right\rbrace = \left\lbrace{\mathnormal{p_{p_{._{._{._{p_{n}}}}}}}}\right\rbrace \; \text{($p$ ``$k$" times)}.
\end{equation*}

\

\noindent Broughan and Barnett \cite{PrimesHavingPrimeSubscripts} have shown that for the general case of higher-order superprimes $\mathnormal{p_{p_{._{._{._{p_{n}}}}}}}$, the asymptotic density is approximately

\begin{equation*}
\mathnormal{\dfrac{n}{p_{p_{._{._{._{p_{n}}}}}}} \sim \dfrac{n}{n\,(\ln{n})^k} \sim \dfrac{1}{(\ln{n})^k} } \;\;\; 
\end{equation*}

\noindent 

\noindent for large $n \in \mathbb{N{}}$. 

\

\noindent Now, assuming that $\ln{n}$ is the asymptotic limit of the gap size for the set of all prime numbers $\mathbb{P{}}$ at $\infty$, we derive an expression for the density $d$ for the prime number subsequence $\mathbb{P{'}}$ at $\infty$.  We begin with the geometric series

\begin{equation*}
\mathnormal{S=1-x+x^2-x^3+x^4-x^5+... = \dfrac{1}{1+x} \;\;\; (\lvert x \rvert < 1)}.
\end{equation*}

\noindent Then let

\begin{align*}
\mathnormal{T} \;\; & \mathnormal{= -S + 1} \\
\end{align*}

\noindent so that

\begin{align*}
\mathnormal{T} \;\; & \mathnormal{= x - x^2 + x^3 - x^4 + x^5+...} \\
& \mathnormal{= \dfrac{-1}{1+x}+1} \\
& \mathnormal{= \dfrac{x}{1+x}}.
\end{align*}

\noindent 

\noindent Now substitute $\dfrac{1}{\ln{n}}$ for $x$ to get

\begin{equation*}
\mathnormal{\dfrac {\dfrac{1}{\ln{n}}}{1+\dfrac{1}{\ln{n}}} = \dfrac{1}{\ln{n}+1}}
\end{equation*}

\noindent so that

\begin{align}
\mathnormal{T} = \mathnormal{\mathnormal{{d_{\mathbb{P{'}}}}}} & \mathnormal{ \approx \dfrac{1}{\ln{n}}-\dfrac{1}{(\ln{n})^2}+\dfrac{1}{(\ln{n})^3}-\dfrac{1}{(\ln{n})^4}+...} \\
& \mathnormal{= \dfrac{1}{\ln{n}+1}}.  \label{eq:6}
\end{align}

\noindent Based on our assumption that $1/\ln{n}$ is the asymptotic limit of the density of the set of all prime numbers $\mathbb{P{}}$ as $n \rightarrow \infty$, Eq. \ref{eq:6} provides us with the density (or probability) of the occurrence of the primes $\mathbb{P{'}}$ at $\infty$.  Thus, the average gap $g$ between prime numbers in the subsequence $\mathbb{P{'}}$ on the natural number line as $n \rightarrow \infty$ is the inverse of the density $d$ of $\mathbb{P{'}}$ so that

\begin{align*}
\mathnormal{\mathnormal{g_{\mathbb{P{'}}}}} = \mathnormal{\frac{1}{d_{\mathbb{P{'}}}}} & \approx \mathnormal{\dfrac{1}{\dfrac{1}{\ln{n}}-\dfrac{1}{(\ln{n})^2}+\dfrac{1}{(\ln{n})^3}-\dfrac{1}{(\ln{n})^4}+...}} \\
& = \mathnormal{\ln{n}+1}.
\end{align*}

\noindent 

\noindent Since it has been shown via the N-sieving operation \cite{HigherOrderPrimeNumberSequences} that the prime number subsequence $\mathbb{P{'}}$ has fewer primes than the set of all prime numbers $\mathbb{P{}}$ at $\infty$, it intuitively follows that the average gap size for $\mathbb{P{'}}$ will always be larger than the gap size for $\mathbb{P{}}$ at $\infty$.

\section{ESTIMATING $\mathnormal{\pi{(x)}}$ VIA $\mathnormal{\mathbb{P{'}}}$}

\noindent 

\noindent When we remove the prime number subsequence $\mathbb{P{''}}$ from the set of all prime numbers $\mathbb{P{}}$, we create the prime number subsequence $\mathbb{P{'}}$ \cite{HigherOrderPrimeNumberSequences}.  Further, it was shown in the previous section that the limit of the average gap size between the prime numbers $\mathbb{P{'}}$ is $\ln{n}+1$ as $n$ tends toward $\infty$.  Thus, the increase made to the average gap between the prime numbers $\mathbb{P{'}}$ when discounting the primes $\mathbb{P{''}}$ on the natural number line is unity at $\infty$.  This is equivalently stated by Eq. \ref{eq:7} wherein the average gap size for the set of all prime numbers $\mathbb{P{}}$ is subtracted from the average gap size for the set of all prime numbers $\mathbb{P{'}}$ to yield the average contribution that the removal of the prime numbers $\mathbb{P{''}}$ adds to to produce the average gap size for $\mathbb{P{'}}$ at $\infty$.

\noindent 

\begin{equation}
\mathnormal{\mathnormal{{g_{\mathbb{P{'}}}-g_{\mathbb{P{}}}}} = 1}. \label{eq:7}
\end{equation}

\noindent 

\noindent Fig. 1 provides a visual representation of the operation of Eq. \ref{eq:7} on the natural number line by showing that in all intervals where an element of $\mathbb{P{''}}$ exists, the removal of $\mathbb{P{''}}$ and replacing with a null integer placeholder iteratively increases the average gap size between the remaining prime numbers $\mathbb{P{'}}$ on the natural number line by unity as that operation is carried out to $\infty$.

\begin{center}
\includegraphics[scale=0.37]{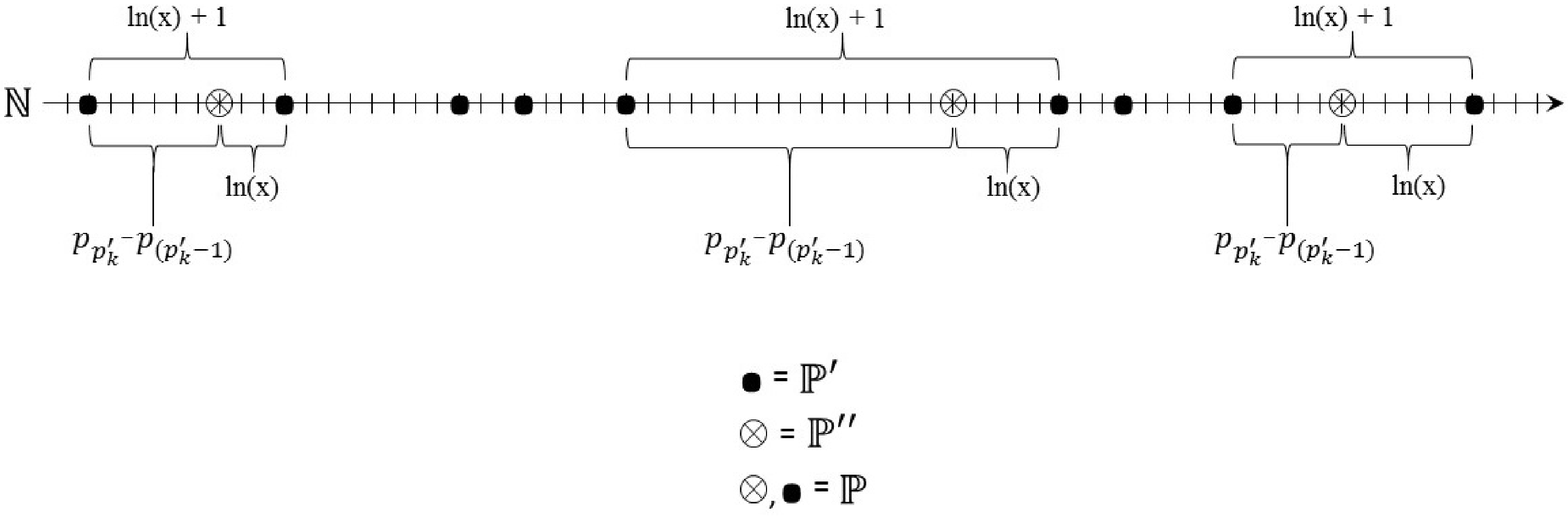}
\end{center}

\begin{center}
\textit{Fig. 1 – Prime Gaps on the Natural Number Line}  
\end{center}

\newpage

\noindent And since

\begin{equation*}
\mathnormal{\mathnormal{\sum\limits_{k=1}^{x} \left( p_{_{p'_{k}}}-p_{_{({p'_{k}-1})}} \right) }}
\end{equation*}

\noindent

\noindent represents the prime gaps \cite{OEIS_A348677} 

\begin{equation*}
\mathnormal{\left\lbrace p_{p'_{n}}-p_{({p'_{n}-1})} \right\rbrace = \left\lbrace 1, 4, 4, 4, 6, 4, 2, 14, 6, 10, 12, 2,...\right\rbrace} 
\end{equation*}

\noindent

\noindent which have been counted as placeholders among the set of all prime numbers $\mathbb{P{}}$ (thereby increasing the gap size between the remaining prime numbers $\mathbb{P{'}}$ from $\ln{n}$ to $\ln{n}+1$ at $\infty$), we arrive at the asymptotic limit

\begin{equation*}
\mathnormal{\mathnormal{{\pi{(x)}}\sim{\sum\limits_{k=1}^{x} \left( p_{_{p'_{k}}}-p_{_{({p'_{k}-1})}} \right) }}},
\end{equation*}

\noindent 

\noindent the sum of which approximates the prime number count $\pi{(x)}$ for the set of all prime numbers $\mathbb{P{}}$ at $\infty$.

\noindent{}

\begin{theorem} \label{1}

The prime number counting function $\pi(x)$ is asymptotically equivalent to an operation performed on a unique subsequence of the prime numbers in that

\begin{equation*}
\mathnormal{\mathnormal{{\pi{(x)}}\sim{\sum\limits_{k=1}^{x} \left( p_{_{p'_{k}}}-p_{_{({p'_{k}-1})}} \right) }}}
\end{equation*}

\noindent 

\noindent which states that the magnitude of the gaps contributed by an operation performed on the unique prime number subsequence $p_{p'}$ as $x$ approaches $\infty$ is asymptotically equivalent to the total number of primes counted by $\pi(x)$ as $x$ approaches $\infty$.

\end{theorem}

\noindent 


\noindent 

\begin{proof}

\noindent We begin with the asymptotic limit of the prime counting function \cite{IntroAnalyticNumberTheory} 

\begin{equation}
\mathnormal{\pi{(x)} \sim \frac{x}{\ln{x}}} \label{eq:8}
\end{equation}

\noindent to show that as $x \rightarrow \infty$,

\begin{equation}
\mathnormal{\lim_{x \to \infty} {\frac{\mathnormal{\pi{(x)}}}{\sum\limits_{k=1}^{x} \left( p_{_{p'_{k}}}-p_{_{({p'_{k}-1})}} \right) } = 1}}. \label{eq:9}
\end{equation}

\noindent 

\noindent In order to evaluate Limit \ref{eq:9}, we need to express both the numerator and denominator in terms of $x$ and $\ln{x}$. The asymptotic limit of the prime counting function in terms of $x$ and $\ln{x}$ is defined in \ref{eq:8}, and a careful examination of Fig. 1 reveals that the denominator of \ref{eq:9} can be expressed as

\begin{equation}
\mathnormal{x\left[1-\frac{\ln{x}}{\ln{x}+1}\right]} \label{eq:11}
\end{equation}

\noindent so that

\begin{equation*}
\mathnormal{\lim_{x \to \infty} {\frac{\mathnormal{\pi{(x)}}}{\sum\limits_{k=1}^{x} \left( p_{_{p'_{k}}}-p_{_{({p'_{k}-1})}} \right) } = \mathnormal{\frac{\frac{x}{\ln{x}}}{x\left[1-\frac{\ln{x}}{\ln{x}+1}\right]} }} = \frac{\ln{x}+1}{\ln{x}}}.
\end{equation*}

\noindent And clearly,

\begin{equation*}
\mathnormal{\lim_{x \to \infty} {\frac{\ln{x}+1}{\ln{x}} = 1}}.
\end{equation*}

\end{proof}

\noindent We also show that the asymptotic limit of the ratio of $\pi{(x)}$ to the complement of the sum in the denominator of Limit \ref{eq:9}, or

\begin{equation}
\mathnormal{\mathnormal{{x}-{\sum\limits_{k=1}^{x} \left( p_{_{p'_{k}}}-p_{_{({p'_{k}-1})}} \right) }}}, \label{eq:12}
\end{equation}

\noindent converges to zero at $\infty$, implying that the complement Expression \ref{eq:12} is infinitely larger than the count of prime numbers for infinitely large x.

\begin{theorem} \label{1}

The prime number counting function $\pi(x)$ is asymptotically equivalent to zero when evaluated against the complement expression

\begin{equation*}
\mathnormal{\mathnormal{{x}-{\sum\limits_{k=1}^{x} \left( p_{_{p'_{k}}}-p_{_{({p'_{k}-1})}} \right) }}} \label{eq:13}
\end{equation*}

\noindent as $x$ approaches $\infty$.

\end{theorem}

\begin{proof}

\noindent We again begin with the asymptotic limit of the prime counting function \cite{IntroAnalyticNumberTheory} in \ref{eq:8} to show that as $x \rightarrow \infty$,

\begin{equation}
\mathnormal{\lim_{x \to \infty} {\frac{\mathnormal{\pi{(x)}}}{x-\sum\limits_{k=1}^{x} \left( p_{_{p'_{k}}}-p_{_{({p'_{k}-1})}} \right) } = 0}}. \label{eq:14}
\end{equation}

\noindent 

\noindent To transform the denominator of the Limit \ref{eq:14} to an expression in terms of $x$ and $\ln{x}$, we recognize that when one subtracts \ref{eq:11} from $x$, we have

\begin{equation*}
\mathnormal{x-x\left[1-\frac{\ln{x}}{\ln{x}+1}\right]=x\left[1-\frac{1}{\ln{x}+1}\right]} \label{eq:15}
\end{equation*}

\noindent so that

\begin{align*}
\mathnormal{\lim_{x \to \infty} {\frac{\mathnormal{\pi{(x)}}}{x-\sum\limits_{k=1}^{x} \left( p_{_{p'_{k}}}-p_{_{({p'_{k}-1})}} \right) }}} & = \mathnormal{\frac{\frac{x}{\ln{x}}}{x\left[1-\frac{1}{\ln{x}+1}\right] }} \\
& \mathnormal{= \frac{\ln{x}+1}{(\ln{x})^2}}.
\end{align*}

\noindent And clearly,

\begin{equation*}
\mathnormal{\lim_{x \to \infty} {\frac{\ln{x}+1}{(\ln{x})^2} = 0}}.
\end{equation*}

\end{proof}

\section{APPROXIMATION OF ${\mathnormal{\pi{(n)}}}$ FOR $\mathnormal{\mathnormal{n < \mathnormal{\infty}}}$}

\noindent 

\noindent It was discovered using Mathematica that the prime number count can be estimated with a notable degree of accuracy (within the bounds of a multiplicative constant) by performing the aforementioned operation on the prime number subsequence of higher-order up to a finite integer $p_{p'_{N}}$, i.e.,

\begin{equation}
\mathnormal{\mathnormal{\pi{(p_{p'_{N}})} \approx C_3*{\sum\limits_{n=1}^{N} \left( p_{_{p'_{n}}}-p_{_{({p'_{n}-1})}} \right) }}}. \label{eq:16}
\end{equation}

\noindent 

\noindent The results of Eq. \ref{eq:16} tabulated below begin at $p_{p'_{N}} \approx 100$ and incrementally go up to $p_{p'_{N}} \approx 10E6$:

\noindent 

\begin{center}
\begin{tabular}{|c|c|c|c|c|c|c|c|}

\hline 
 $p_{p'_{N}}$ & $\pi(p_{p'_{N}})$ & ${\sum\limits_{n=1}^N p_{p'_{n}}-p_{({p'_{n}-1})}}$ & $C_3$ \\
\hline 
1E02 & 25 & 23 & 1.08696  \\
\hline 
1E03 & 168 & 187 & 0.89840  \\
\hline 
1E04 & 1,229 & 1,319 & 0.93177  \\
\hline 
1E05 & 9,592 & 10,651 & 0.90057  \\
\hline 
1E06 & 78,498 & 86,249 & 0.91013  \\
\hline 
2E06 & 148,933 & 165,133 & 0.90190  \\
\hline 
3E06 & 216,816 & 239,893 & 0.90380  \\
\hline 
4E06 & 283,146 & 312,563 & 0.90588  \\
\hline 
5E06 & 348,513 & 384,277 & 0.90693  \\
\hline 
6E06 & 412,849 & 455,401 & 0.90656  \\
\hline 
7E06 & 476,648 & 525,917 & 0.90632  \\
\hline 
8E06 & 539,777 & 595,285 & 0.90675  \\
\hline 
9E06 & 602,489 & 665,345 & 0.90553  \\
\hline 
10E6 & 664,579 & 733,389 & 0.90618  \\
\hline 

\end{tabular}
\end{center}

\noindent 

$ $

\noindent 

\noindent An observance of the data in the table above reveals that the constant $C_3$ appears to oscillate rather tightly around the value

\begin{equation*}
\mathnormal{\frac{\pi{}\sqrt{3}}{6}}
\end{equation*}

\noindent 

\noindent which happens to be the densest packing density possible for identi-cally-sized circles in a plane. This would imply (at least within the range of $p_{p'_{N}}$ in the table) that the ratio of the prime counting function $\pi{(p_{p'_{N}})}$ to the sum of the gaps counted by

\begin{equation*}
\mathnormal{{\sum\limits_{n=1}^{N} \left( p_{_{p'_{n}}}-p_{_{({p'_{n}-1})}} \right)}}
\end{equation*}

\

\noindent closely approximates the density of identical circles packed as tightly as possible in a hexagonal packing arrangement in a plane. More study is needed to determine the convergence or divergence of the constant $C_3$. 

\newpage

\section{CONCLUSION}

\noindent In this paper, we derived an expression for the asymptotic limit of the prime-counting function $\mathnormal{\pi{(x)}}$ as a function of the prime number subsequence of higher order $\mathbb{P{'}}$. We further showed that the expression derived is a good approximation (to within a constant $C_3$) of the prime counting function $\pi(n)$ up to any positive real $N \le 10E6$.



\noindent

\noindent

\noindent

\noindent

\end{document}